\documentclass[10pt, english,oneside,reqno]{amsart}
\usepackage[utf8]{inputenc}
\usepackage[english]{babel}
\usepackage[
    colorlinks=false,
    pdfborder={0 0 1},
    linkbordercolor={1 0 0},
    citebordercolor={0 1 0},
    urlbordercolor={0 1 1}
]{hyperref}
\usepackage{color}
\usepackage{enumerate}
\usepackage{fancyhdr}

\usepackage{amsmath}
\usepackage{amscd}
\usepackage{amsfonts}
\usepackage{amsthm}
\usepackage{amstext}
\usepackage{amssymb}
\usepackage{amsbsy}
\usepackage{mathrsfs}
\usepackage{tikz-cd}
\usepackage{mathtools}
\usepackage{enumitem}
\newtheorem*{corollary*}{Corollary}

\usepackage{float}
\usepackage{caption}
\usepackage{subcaption}
\usepackage{graphicx}

\usepackage[alphabetic,abbrev]{amsrefs}

\usepackage{url}

\theoremstyle{plain}
\newtheorem{theorem}{Theorem}[section]

\newtheorem{definition}[theorem]{Definition}
\newtheorem{proposition}[theorem]{Proposition}
\newtheorem{corollary}[theorem]{Corollary}
\newtheorem{lemma}[theorem]{Lemma}

\theoremstyle{remark}

\numberwithin{equation}{section}
\numberwithin{figure}{section}

\newcommand{\E}{\mathbb{E}}
\newcommand{\Prob}{\mathbb{P}}
\newcommand{\R}{\mathbb R}

\newcommand{\cL}{\mathcal L}

\newcommand{\ind}{\mathbf 1}

\DeclareMathOperator{\supp}{supp}
\DeclareMathOperator{\Var}{Var}

\usepackage{microtype}
\usepackage{needspace}
\newcommand{\Energy}{\mathcal E}
\newcommand{\cE}{\mathcal E}
\newcommand{\cA}{\mathcal A}
\newcommand{\Disp}{\mathcal D}
\newcommand{\CPE}{\mathcal A}
\newcommand{\Z}{\mathbb Z}

\hypersetup{pdftitle={The generalized exact overlaps conjecture on the dimension of self-similar measures on the line}}
\title[The exact overlaps conjecture on the real line]
{The exact overlaps
conjecture for
self-similar measures on the real line}

\begin{document}
\author[sk]{Samuel Kittle}
\author[ck]{Constantin Kogler}

\begin{abstract}
A generalization of the exact overlaps conjecture predicts a formula for the dimension of self-similar measures in terms of the random walk entropy and the Lyapunov exponent. We prove the exact overlaps conjecture and its generalized version in dimension one. The proof relies on a new entropy inequality quantifying the information lost under addition of independent random variables in terms of a quantity averaging variance across scales.
\end{abstract}

\email{s.kittle@ucl.ac.uk, kogler@ias.edu}

    \address{Samuel Kittle, Department of Mathematics, University College London, 25 Gordon Street, London WC1H 0AY, United Kingdom}

    \address{Constantin Kogler, Institute for Advanced Study, 1 Einstein Dr, Princeton, NJ 08540, United States of America}
\maketitle

\section{Introduction}
Determining the dimension of self-similar measures is a fundamental problem in fractal geometry. In this paper, we address this question in dimension one.
For $1 \leq i \leq k$ and $x \in \mathbb{R}$ let $g_i(x) = r_i x + b_i$ with $0<|r_i|<1$ and $b_i\in \mathbb{R}$ be finitely many distinct contracting similarities on $\mathbb{R}$. We consider a probability measure $\mu = \sum_{i = 1}^k p_i \delta_{g_i}$ supported on the given similarities with $p_i > 0$ for all $1\leq i \leq k$. It is well-established \cite{Hut} that there exists a unique $\mu$-stationary probability measure $\nu$ on $\mathbb{R}$, that is, a measure satisfying
\begin{equation*}
 \mu*\nu=\sum_{i=1}^k p_i(g_i)_*\nu = \nu.
\end{equation*}
The measure $\nu$ is called the self-similar measure of $\mu$. In \cite{FH} it is shown that $\nu$ is
exact dimensional, meaning that there exists $\alpha \in[0,1]$ such that for $\nu$-almost all $x \in \mathbb{R}$ it holds that $\nu([x-r,x+r]) = r^{\alpha + o(1)}$ as $r \to 0$. We refer to $\alpha$ as the dimension of $\nu$ and write $\dim\nu=\alpha$.

The fundamental quantities to calculate the dimension of $\nu$ are the Lyapunov exponent $\chi_{\mu}$ and the random walk entropy $h_{\mu}$ given by 
\[
\chi_{\mu}=\sum_i p_i\log|r_i| \quad\text{ and }\quad h_{\mu}=\lim_{n\to\infty}\frac{H(\mu^{*n})}n,
\]
where $H$ denotes the Shannon entropy. It is well-known that $\dim\nu\le\min\{1,\frac{h_{\mu}}{|\chi_{\mu}|}\}$
and conjectured for example in \cite{VS}*{Conjecture 3} that equality always holds in the latter inequality, that is \begin{equation}\label{eq:exactoverlaps}
 \dim\nu=\min\left\{1,\frac{h_{\mu}}{|\chi_{\mu}|}\right\}.
\end{equation} 

The exact overlaps conjecture, which is mentioned for example in~\cite{B2}, \cite{R}, \cite{FF}, and \cite{W}, claims for $H(p) = H(\mu) = -\sum_{i = 1}^k p_i\log p_i$ that $\dim\nu<\min\{1,\frac{H(p)}{|\chi_{\mu}|}\}$ only if $\mu$ has exact overlaps, meaning that $\mathrm{supp}(\mu)$ does not freely generate a semigroup. As $h_{\mu} = H(p)$ if there are no exact overlaps, the conjecture \eqref{eq:exactoverlaps} is a generalization of the exact overlaps conjecture.
A version of this conjecture for self-similar sets dates back to Simon from 1996 \cite{S}.

A breakthrough on the exact overlaps conjecture was achieved by Hochman \cite{H}, establishing the conjecture under
an exponential separation condition and thereby proving the conjecture for similarities with algebraic parameters. As \cite{H} leaves the case of transcendental parameters open, there has been substantial effort to establish Hochman's results also for transcendental parameters. In \cite{HICM}, Hochman asked if failure of exponential separation necessarily implies exact overlaps, which was disproven by Baker~\cite{B} and B\'ar\'any and K\"aenm\"aki~\cite{BK}. Shmerkin \cite{Shm19} extended Hochman's results to $L^q$ dimensions and as a remarkable application proved Furstenberg's intersection conjecture.

An important achievement was Varj\'u \cite{V} proving the exact overlaps conjecture for transcendental Bernoulli convolutions, relying on delicate results in Diophantine approximation combined with work of
Breuillard and Varj\'u \cite{BVdim},\cite{BVent}. Rapaport \cite{R} settled systems with algebraic contraction ratios and
arbitrary translations. Rapaport and Varj\'u
\cite{RV} proved further cases for homogeneous systems of three maps. Indeed, for the maps $x\mapsto \lambda x, x\mapsto \lambda x + 1$ and $x \mapsto \lambda x + \tau$ with $\lambda \in (0,1)$ and $\tau \in \mathbb{R}$ and $\mu$ the probability measure on these three maps with equal probability weights, Rapaport and Varj\'u
\cite{RV} proved, for example, the exact overlaps conjecture if $(\lambda,\tau) \in (2^{-2/3},1)\times \mathbb{R}$. Moreover, \cite{RV} reduced the exact overlaps conjecture for this system of three maps to a criterion involving the Mahler measure of $\lambda$ and showed that the set of exceptions for $(\lambda, \tau) \in (0,1)\times \mathbb{R}$ has Hausdorff dimension 0. In the thesis of Streck \cite{Str24}, the exact overlaps conjecture was established for the equal weight measure of the maps $x\mapsto \lambda x, x\mapsto \lambda x + 1,x\mapsto \lambda x + \tau$ and $x\mapsto \lambda x + 1 + \tau$ with $\lambda \in (0,1)$ and $\tau \in \mathbb{R}$ if $(\lambda,\tau) \in (2^{-3/2},1) \times \mathbb{R}$.

Feng and Feng \cite{FF} treated
algebraic translations with an arbitrary nonzero real common contraction. Wu \cite{W} showed under a transversality condition that the set of exceptions for the exact overlaps conjecture is at most countable, without making the set of exceptions explicit. We also mention that the variance summation method, as initiated by the first-named author in \cite{Kfur}, was used by the authors \cite{KK25} to generalize Hochman's theorem to contracting on average measures and to show that a weaker condition than exponential separation is sufficient. 

In this paper we establish the exact overlaps conjecture in dimension one as well as the more general conjecture \eqref{eq:exactoverlaps}.

\begin{theorem}\label{thm:main}
Let $\mu$ be a probability measure supported on finitely many contracting similarities of $\mathbb{R}$. Then the self-similar measure $\nu$ of $\mu$ satisfies
\[
 \dim\nu=\min\left\{1,\frac{h_{\mu}}{|\chi_{\mu}|}\right\}.
\]
\end{theorem}

A proof of Theorem~\ref{thm:main} was independently announced in \cite{OpenAIMath} as paper number 148. We will not discuss the mathematical content of the latter paper in detail, yet remark that Theorem~\ref{prop:addition} stated below is a quantitative version of Lemma 3.2 of \cite{OpenAIMath}*{Paper 148} and therefore a strengthening. We highlight that Theorem~\ref{prop:addition} and Lemma 3.2 of \cite{OpenAIMath}*{Paper 148} are, when applied to self-similar measures, the core novelty of the proofs. In this paper, we used ideas from \cite{Kfur}, \cite{KKabs} and \cite{KK25} to conclude the proof of Theorem~\ref{thm:main}, whereas \cite{OpenAIMath}*{Paper 148} uses entropy-growth arguments in the spirit of \cites{H, Vabs}.

We refer the reader to our machine contribution statement explaining how the authors found an initial proof of Theorem~\ref{thm:main} and discussing the authors' contribution. We remark that \cite{OpenAIMath} claim that their proof of Theorem~\ref{thm:main} was found on September 24, 2026, while the authors first discovered a proof of Theorem~\ref{thm:main} on September 27, 2026. \cite{OpenAIMath} became public on October 6, 2026 and the authors made this paper available online on October 7, 2026.

To give a concrete consequence of Theorem~\ref{thm:main}, consider for $x \in \mathbb{R}$ the similarities $$g_1(x)= \lambda_1 x + 1 \quad\quad\text{ and }\quad\quad  g_2(x) = \lambda_2 x - 1$$ with $\lambda_1,\lambda_2 \in (0,1)$ and let $\mu = \frac{1}{2}\delta_{g_1} + \frac{1}{2}\delta_{g_2}$ be the unbiased probability measure supported on these two maps. Suppose that, after interchanging indices if necessary, $\lambda_2$ is transcendental over $\mathbb{Q}(\lambda_1)$. Then the maps have no exact overlaps and therefore it follows from Theorem~\ref{thm:main} that $$\dim \nu = \min\left\{ 1,\frac{2\log 2}{\log \lambda_1^{-1} + \log \lambda_2^{-1}} \right\}.$$

Given contracting similarities $g_1, \ldots, g_k$ on $\mathbb{R}$, there exists a unique non-empty compact set $K \subset \mathbb{R}$ satisfying $K = \bigcup_{i = 1}^k g_i(K)$. The set $K$ is called the self-similar set or the attractor of $g_1, \ldots , g_k$. A significant motivation for the exact overlaps conjecture is, as is well-known, that it implies an explicit formula for the Hausdorff dimension of $K$ when there are no exact overlaps.
Under exponential separation, this dimension formula follows from \cite{H}*{Corollary 1.2}. Theorem~\ref{thm:main} implies the formula assuming only the absence of exact overlaps. 

\begin{corollary}\label{cor:self-similar-sets}
Let $g_1,\ldots , g_k$ be contracting similarities on $\mathbb{R}$ with contraction rates $r_i$ satisfying $0 < |r_i| < 1$ for $1 \leq i \leq k$. Let $s\geq 0$ be the unique non-negative solution of $\sum_{i=1}^k |r_i|^s=1$. If $g_1,\ldots , g_k$ have no exact overlaps, then the self-similar set $K$ satisfies
\[
    \dim K=\min\{1,s\}.
\]
\end{corollary}

The proof of Theorem~\ref{thm:main} presented in this paper relies on the entropy inequality Theorem~\ref{prop:addition}, which is of independent interest, combined with \cite{H}*{Theorem 1.4} and a simplified version of the variance summation method from \cite{KK25}.


As shown in \cite{KK26Supp}, the method of this paper extends to arbitrary dimensions provided that the rotations of the similarities in \(\operatorname{supp}\mu\) generate a virtually solvable group, which, in particular, includes the two dimensional case in full generality. By the Tits alternative \cite{Tit72}, the latter condition is equivalent to requiring that the entropy of the linear part of the product of the first \(n\) similarities is \(o(n)\). However, the application of the entropy inequality Theorem \ref{prop:addition} in Proposition \ref{prop:finite-sum} as discussed below fundamentally relies on addition being an abelian operation. Therefore, we do not regard this paper as having made significant progress on the general $d$-dimensional case.

\subsection*{A new entropy inequality}

 The main novelty in the proof of Theorem~\ref{thm:main} is a general inequality estimating how much the entropy of a sum of random variables can differ from the entropy of the summands. We first state the new inequality and then discuss how it can be combined with previous ideas to deduce Theorem~\ref{thm:main}.

To state the entropy inequality, we first introduce a few definitions. For a compactly supported probability measure $\rho$ on $\mathbb{R}$ and a scale $r > 0$ we denote by $V_{\rho}(r)$ the normalized local variance defined as follows. Write $I(a,r)=[a,a+r)$ and for $r > 0$ we define
\begin{equation}\label{eq:local-variance}
 V_\rho(r)=\frac4{r^3}\int_\R F_\rho(I(a,r))\,da,
\end{equation} where $$F_\rho(I(a,r))=\rho(I(a,r))\Var(\rho\mid I(a,r))
          =\inf_{c\in\R}\int_{I(a,r)}(x-c)^2\,d\rho(x).$$
The factor $4r^{-3}$ consists of $r^{-1}$ for averaging over the size of the interval, $r^{-2}$ for normalizing the squared error and the constant $4$ is chosen so that always $0\leq V_{\rho}(r) \leq 1$. We furthermore define for $R > 0$
\begin{equation}\label{eq:energy}
 \Energy_R(\rho)=\int_0^R V_\rho(r)\,\frac{dr}{r}
 \quad \text{ and } \quad \Energy(\rho)=\Energy_\infty(\rho).
\end{equation}
We call $\mathcal{E}(\rho)$ the variance energy. The logarithmic measure $dr/r$
assigns the same weight to every interval of scales $[r,2r]$. We
write $V_Y(r)$ for $V_\rho(r)$ and $\mathcal{E}(Y)$ for $\mathcal{E}(\rho)$ when the random variable $Y$ has law $\rho$.

We now state the entropy inequality. Note that we call a random variable finitely supported if its law has finite support.

\begin{theorem}[Entropy loss under addition]\label{prop:addition}
For $m\geq 1$ and independent finitely supported real random variables $X_1,\ldots,X_m$, it holds that
\begin{equation}\label{eq:addition}
 \sum_{i=1}^m H(X_i)-H\left(\sum_{i=1}^mX_i\right)
 \le 30 m\sum_{i=1}^m\Energy(X_i).
\end{equation}
\end{theorem}

To the authors' knowledge, Theorem~\ref{prop:addition} is novel. The inequality claims that addition of random variables can only conceal a lot of information if the summands have sufficient variance across scales. The key idea of the proof is to reveal more and more information of $X_i$ by partitioning the space with a Poisson point process and to apply Margulis-Russo formulas on the derivative of Poisson processes. A sketch of the proof of Theorem~\ref{prop:addition} together with a discussion of the related literature is given at the beginning of Section~\ref{sec:addition}.

We highlight that Theorem~\ref{prop:addition} is far more powerful than necessary to prove Theorem~\ref{thm:main} as any bound of the form $\leq C_m\sum_{i=1}^m\Energy(X_i)$ for some finite constant $C_m$ depending only on $m$ would be sufficient. A bound of this kind with an unspecified constant $C_m$ was established in Lemma 3.2 of \cite{OpenAIMath}*{Paper 148} and we remark that Theorem~\ref{prop:addition} implies a strengthening of Lemma 3.2 of \cite{OpenAIMath}*{Paper 148} with the explicit constant $C_m = 120m^2$. The proof method of Lemma 3.2 of \cite{OpenAIMath}*{Paper 148} leads to a constant of the form $C_m = O(m^{3})$ and, while Lemma 3.2 of \cite{OpenAIMath}*{Paper 148} does not imply Theorem~\ref{prop:addition}, the proof method of Lemma 3.2 of \cite{OpenAIMath}*{Paper 148} can be adapted to prove Theorem~\ref{prop:addition} with the constant $C_m = O(m^3)$. In \cite{KK26Supp} it is shown that a sharpening of the proof of Theorem~\ref{prop:addition} leads to the bound $\leq 62\log m\sum_{i=1}^m\Energy(X_i)$ in \eqref{eq:addition} and that the logarithmic dependence on \(m\) cannot be improved.

We next discuss how to use Theorem~\ref{prop:addition} to prove Theorem~\ref{thm:main}.

\subsection*{Deduction of Theorem~\ref{thm:main} from Theorem~\ref{prop:addition}} For simplicity, we explain the argument for homogeneous self-similar measures. The general case can be deduced with analogous ideas by using appropriate conditioning on the contraction rate. So consider for $1 \leq i \leq k$ and $x \in \mathbb{R}$ the maps $g_i(x)=\lambda x+b_i$ with $\lambda \in (0,1)$ and $b_i \in \mathbb{R}$. A sample from the self-similar measure can by stationarity be written as
\[
Z=\sum_{j=0}^{\infty}\lambda^jD_j
  =B_n+\lambda^nZ' \quad \text{ with } 
\quad
B_n=\sum_{j=0}^{n-1}\lambda^jD_j,
\]
where the random variables \(D_j\) are independent and each takes the value \(b_i\) with probability \(p_i\) and \(Z'\) has law $\nu$ and is independent of $B_n$. Thus the measure decomposes into a finite block of independent digits and an independent rescaled copy of itself. Note that in the homogeneous case $\chi_{\mu} = \log \lambda$.

To prove Theorem~\ref{thm:main} we argue by contradiction and assume that $\dim \nu < \min\{1, h_{\mu}/\log \lambda^{-1}\}$. We will show this leads to the contradiction that $\dim \nu = 1$. By the definition of $h_{\mu}$, the Shannon entropy of \(B_n\) is \(nh_\mu+o(n)\). On the other hand, since $\dim \nu < 1$ by \cite{H}*{Theorem 1.4}, at any fixed exponential scale \(R = C^{-n}\) with \(C>\lambda^{-1}\), it holds that
\[
H(B_n;C^{-n})=n\log\lambda^{-1}\dim \nu+o(n),
\] where $H(B_n;C^{-n})$ is the entropy at scale $R = C^{-n}$ from \cite{Vabs} as recalled in Section~\ref{sec:proof_main}.
Thus it follows for $\delta=h_\mu-\log\lambda^{-1}\dim \nu > 0$ that we have the entropy gap
\begin{equation}\label{entgap}
H(B_n)-H(B_n;C^{-n})=\delta n+o(n).
\end{equation}
In other words, even after identifying the interval of size $R = C^{-n}$ containing \(B_n\), a linear amount of information remains unresolved inside that interval.

Theorem~\ref{prop:addition} converts this unresolved information into variance at scales below \(R= C^{-n}\). To see how to apply Theorem~\ref{prop:addition}, we introduce \(m\) independent copies of \(B_n\) that we denote by $B_n^{(1)}, \ldots ,B_n^{(m)}$. Their sum has the decomposition
\[
\sum_{\ell=1}^{m}B_n^{(\ell)}
=
\sum_{j=0}^{n-1}\lambda^j
\left(\sum_{\ell=1}^{m}D_j^{(\ell)}\right).
\]
If each digit has at most \(q\) possible values, this gives
\begin{equation}\label{trivialentbound}
H\!\left(\sum_{\ell=1}^{m}B_n^{(\ell)}\right)
\le n(q-1)\log(m+1).
\end{equation}
Thus the entropy of the sum, divided by the number of copies, becomes small compared with \(n\) when \(m\) is large.

Apply Theorem 1.3 after conditioning each copy on its interval of length \(R = C^{-n}\). Now take a grid consisting of intervals of length $R = C^{-n}$ and fix the origin of the grid. Let \(Q_\ell\) label the interval containing \(B_n^{(\ell)}\), with \(Q\) denoting the label for $B_n$. Averaging the entropy inequality Theorem~\ref{prop:addition} over these labels therefore gives
\begin{equation}\label{additionbound}
mH(B_n\mid Q)
-H\!\left(\sum_{\ell=1}^m B_n^{(\ell)}\,\middle|\,Q_1,\ldots,Q_m\right)
\le 30m^2\,\E_Q\bigl[\Energy(B_n\mid Q)\bigr].
\end{equation}
Moreover, each conditional law is supported on an interval of length \(R = C^{-n}\), so a tail bound and concavity of \(\Energy_R\) give
\[
\E_Q\bigl[\Energy(B_n\mid Q)\bigr]
\le \E_Q\bigl[\Energy_{C^{-n}}(B_n\mid Q)\bigr]+\frac12
\le \Energy_{C^{-n}}(B_n)+\frac12.
\]
Finally, averaging over the grid origin turns \(H(B_n\mid Q)\) into \(H(B_n)-H(B_n;C^{-n})\). Applying \eqref{additionbound} and \eqref{trivialentbound} and then dividing by \(m\) yields
\[
H(B_n)-H(B_n;C^{-n})
\le
\frac{n(q-1)\log(m+1)}{m}
+30m\,\mathcal E_{C^{-n}}(B_n)+15m.
\]
Choose \(m\) sufficiently large and then let \(n\) tend to infinity. The entropy gap \eqref{entgap} therefore shows
\[
\mathcal E_{C^{-n}}(B_n)\ge cn
\]
for some \(c>0\) and sufficiently large $n$. Crucially, we gain a multiscale variance bound. At the beginning of Section~\ref{sec:proof_main}, we explain how to use this bound to conclude the proof. Indeed, we combine $\mathcal E_{C^{-n}}(B_n)\ge cn$ with a simplified version of the variance summation method from \cite{KK25} to deduce $\dim \nu = 1$, which is a contradiction.

\subsection*{Machine Contribution Statement} On September 27, 2026, GPT-6 Astra found the first proof of Theorem~\ref{thm:main} presented to the second-named author (C.K.). Motivated by a suggestion of Elon Lindenstrauss, GPT-6 Astra was first asked to propose new approaches to simpler proofs of the main results of \cite{H} and \cite{V}, which was successfully achieved after around two hours. Encouraged by the model having apparently rethought the area in two hours, the exact overlaps conjecture was tried and the model succeeded after around two and a half hours. The initial proof included Theorem~\ref{prop:addition} and the deduction of Theorem~\ref{thm:main} from Theorem~\ref{prop:addition} was essentially self-contained. The conversation transcript and the initial model drafts are available at \cite{KK26Supp}.

The second-named author (C.K.) immediately shared the discovery with the first-named author (S.K.), who found a new proof of the deduction of Theorem~\ref{thm:main} from Theorem~\ref{prop:addition} using \cite{H}*{Theorem 1.4} and the variance summation method from \cite{Kfur} and \cite{KK25} while the second-named author (C.K.) was rewriting the rest of the paper. The authors chose to write this new proof in the paper in order to make the paper more easily accessible to experts and to the authors themselves. Therefore, the introduction and Section~\ref{sec:proof_main} are due to the authors, while Section~\ref{sec:addition} is due to GPT-6 Astra and was rewritten by the authors. The connection of the proof of Theorem~\ref{prop:addition} to Margulis-Russo formulas as discussed in Section~\ref{sec:addition} was made by the authors. GPT-6 Astra also assisted with proofreading and stylistic suggestions.

The authors take full responsibility for the correctness of this paper and would be delighted to give talks on the results and to answer any questions.

\subsection*{Lean Formalization}  GPT-6 Astra formalized all results from this paper and all necessary results from previous work including 
\cite{H}*{Theorem 1.4}. The formalization is around 47,000 lines of code and took around 9.5 hours with four agents working in parallel. The Lean source code is available in the accompanying repository \cite{KK26Lean}.

\subsection*{Supplementary Material} We provide supplementary material in \cite{KK26Supp} containing the following:
\begin{enumerate}
  \item An 8-page document with a self-contained proof of \cite{H}*{Theorem 1.1}.
  \item A 10-page document with a self-contained proof of \cite{V}*{Theorem 3}.
  \item A 16-page document with a proof of Theorem~\ref{thm:main} that includes Theorem~\ref{prop:addition} and concludes the proof of Theorem~\ref{thm:main} using ideas similar to (1).
  \item The conversation transcript leading to initial versions of (1), (2), and (3). Minor changes to (1), (2), and (3) were performed after the initial session on September 27, 2026. 
  \item A document generalizing Theorem~\ref{thm:main} to arbitrary dimensions assuming that the rotations of $\mu$ generate a virtually solvable subgroup. This document was generated on the basis of a draft of this paper and contains no significant new ideas.
  \item A document sharpening Theorem~\ref{prop:addition} to the bound $\leq 62 \log m\sum_{i = 1}^m \mathcal{E}(X_i)$ and the logarithmic dependence on \(m\) cannot be improved. This document was generated on the basis of a draft of this paper and contains no significant new ideas.
\end{enumerate}

\subsection*{Acknowledgments} We thank Emmanuel Breuillard, Michael Hochman and Péter Varjú for advice. The first-named author (S.K.) gratefully acknowledges support from the Heilbronn Institute for Mathematical Research. The second-named author (C.K.) holds a Postdoc Mobility Fellowship from the Swiss National Science Foundation (grant number
235409) and thanks the Institute for Advanced Study.
\section{Proof of Theorem~\ref{prop:addition}: Information lost under addition}\label{sec:addition}

In this section we prove Theorem~\ref{prop:addition} and first give an outline of the proof. Let $X_1,\ldots,X_m$ be independent finitely
supported real random variables and write $S=\sum_{i=1}^m X_i$.
By independence,
\[
 \sum_{i=1}^m H(X_i)-H(S)=H(X_1,\ldots,X_m\mid S).
\]
Thus the entropy loss under addition is the uncertainty about the
individual summands that remains after their sum is known. We estimate
this uncertainty by revealing the summands gradually.

For each summand $X_i$, we use an independent family of refining
interval partitions. Indeed, for each $i$ place an independent Poisson point process on $\R\times(0,\infty)$
with intensity $da\,dt$, interpreting a point $(a,t)$ as a cut at
location $a$ introduced at time $t$. For a random variable $Y$ and $P_t$ the partition
formed by all cuts introduced by time $t$ arising from an independent Poisson process we denote by $P_t(Y)$
the cell containing $Y$. Note that the partition given by $P_t$ contains infinitely many points, yet they have no finite accumulation point by the Poisson property. At time $t>0$, there is an average of $t$ cuts per unit length and the gaps between consecutive cuts are exponentially distributed
with mean $1/t$. The partitions are independent of the input variables.
Expressions such as $\E_{P_t}[f(Y\mid P_t(Y))]$
average over both the partitions and the cell containing $Y$.

Using independent partitions
$P_t = (P_t^{(1)},\ldots,P_t^{(m)})$, define
\[
 J(t)=\E_{P_t}\left[H\bigl(P_t^{(1)}(X_1),\ldots,
        P_t^{(m)}(X_m)\mid S\bigr)\right].
\]
Initially, the partitions reveal no information, so $J(0)=0$.
Eventually, they distinguish every point in each finite support.
Consequently,
\[
 \lim_{t\to\infty}J(t)
 =H(X_1,\ldots,X_m\mid S).
\]
It therefore suffices to bound the derivative of $J$ and integrate
over time. Expressions for the derivative of point processes are well-known and referred to in the literature as Margulis-Russo formulas or as perturbation formulas. Such differentiation formulas appeared first for Bernoulli product measures in the work of Esary and Proschan \cite{EP63} and were later independently rediscovered by Margulis \cite{Mar74} and Russo \cite{Rus81} in the study of graph connectivity and percolation. The extension to Poisson processes was established by Zuyev \cite{Zuy92} and we also refer to \cite{LP17}*{Chapter 19} for a systematic exposition. We remark that the remainder of the proof appears to be novel, yet may be considered more natural from the viewpoint of Margulis-Russo formulas.

Indeed, as explained in the proof of Theorem~\ref{prop:addition}, a direct consequence of \cite{LP17}*{Theorem 19.1} is that
\begin{equation}\label{eq:poisson-derivative}
 J'(t)=\E_{P_t}\left[\sum_{i=1}^m\sum_{C \in \mathcal{C}_t}
 \Prob[X\in C]\int_{C_i}
 H(\ind_{\{X_i\le a\}}\mid S,X\in C)\,da\right],
\end{equation}
where $C=C_1\times\cdots\times C_m$ ranges over the partition $\mathcal{C}_t= P_t^{(1)} \times \ldots \times P_t^{(m)}$ and
$X=(X_1,\ldots,X_m)$. Equation \eqref{eq:poisson-derivative} motivates the conditional cumulative entropy defined for random variables $X$ and $Y$ as
\[
 \CPE(X\mid Y)=\int_\R H(\ind_{\{X\le a\}}\mid Y)\,da.
\]
It measures the total uncertainty in threshold tests of $X$ after
$Y$ is known and will be applied with the random variable $S$ conditioned on $C$ in \eqref{eq:poisson-derivative}. We will show in Lemma~\ref{lem:cpe} for $X$ and $Y$ independent finitely supported real random variables, that
\begin{equation}\label{eq:cpe_draft}
 \CPE(X\mid X+Y)\le9\bigl(\Disp(X)+\Disp(Y)\bigr),
\end{equation} where the dispersion $\Disp(X)$ is defined as
\[
 \Disp(X)=\E[|X-X'|],
 \qquad
 \text{for $X'$ is an independent copy of $X$.}
\]

A calculation applying \eqref{eq:cpe_draft} to \eqref{eq:poisson-derivative} leads to the inequality \begin{equation}\label{eq:poisson-bound-draft}
 J'(t)\le9m\sum_i\E_{P_t}\left[\Disp(X_i\mid P_t^{(i)}(X_i))\right].
\end{equation}
Finally, we establish in Lemma~\ref{lem:dispersion} for any random variable $Y$ with independent Poisson point process $P_t$ that
\begin{equation}\label{eq:disp-bound-draft}
\int_0^\infty\E_{P_t}\left[\Disp(Y\mid P_t(Y))\right]\,dt
 \leq 3\Energy(Y).
 \end{equation}
Integrating \eqref{eq:poisson-bound-draft} combined with the bound \eqref{eq:disp-bound-draft} shows 
\[
 \sum_{i=1}^m H(X_i)-H(S)
 =\int_0^\infty J'(t)\,dt
 \le 27m\sum_{i=1}^m \Energy(X_i),
\]
which implies Theorem~\ref{prop:addition} as $27 \leq 30$.

We give the details after proving Lemmas~\ref{lem:cpe}
and~\ref{lem:dispersion}. We begin with some basic properties
of variance energy.

\begin{lemma}\label{lem:variance-basic}
The functionals $F_\rho(I)$, $V_\rho(r)$ and $\Energy_R(\rho)$ are
concave in $\rho$. Moreover, $0\le V_\rho(r)\le1$ and, for $a,b \in \mathbb{R}$ with $a\ne0$,
\[
 V_{aY+b}(r)=V_Y(r/|a|)\quad \text{ and } \quad \Energy(aY+b)=\Energy(Y).
\]
If $\rho$ is finitely supported and its support has diameter at most $D$,
then $\Energy(\rho)<\infty$ and
\begin{equation}\label{eq:energy-tail}
 \Energy(\rho)-\Energy_R(\rho)\le\frac{D^2}{2R^2}.
\end{equation}
\end{lemma}

\begin{proof}
$F_\rho(I)$ is defined by an infimum over integrals over $\rho$ and is therefore concave. The same follows therefore for
$V_\rho(r)$ and $\Energy_R$. Taking the midpoint $c=a+r/2$ in the definition of $F_\rho(I(a,r))$ gives
\[
 0\le F_\rho(I(a,r))\le\frac{r^2}{4}\rho(I(a,r)).
\]
Since $\int_\R\rho(I(a,r))\,da=r$, the factor $4/r^3$ in
\eqref{eq:local-variance} gives $0\le V_\rho(r)\le1$. Substitution gives the affine change of variables formula.

If the support of $\rho$ has diameter at most $D$, then
every conditional law $\rho\mid I$ also has support of
diameter at most $D$, so $\Var(\rho\mid I)\le D^2/4$. Since
$\int_\R\rho(I(a,r))\,da=r$, we have $V_\rho(r)\le D^2/r^2$.
Hence
\[
 \Energy(\rho)-\Energy_R(\rho)
 \le D^2\int_R^\infty\frac{dr}{r^3}
 =\frac{D^2}{2R^2}.
\]
For a finitely supported law, $V_\rho(r)$ also vanishes below the smallest
spacing between atoms. Thus the full energy $\mathcal{E}(\rho)$ is finite. 
\end{proof}

Now we prove Lemma \ref{lem:cpe}.

\begin{lemma}[Conditional cumulative entropy]\label{lem:cpe}
If $X,Y$ are independent finitely supported real random variables, then
\begin{equation}\label{eq:cpe}
 \CPE(X\mid X+Y)\le9\bigl(\Disp(X)+\Disp(Y)\bigr).
\end{equation}
\end{lemma}

\begin{proof}[Proof of Lemma \ref{lem:cpe}]
We begin the proof with the following inequality. Suppose that a random variable $Z$ satisfying $Z\in[0,1]$ almost surely
dominates a uniform variable. Write $T=\E[1-Z]$ and we then claim that
\begin{equation}\label{eq:quantile}
 \E\left[\sqrt{\frac{1-Z}{Z}} \right]\le3\sqrt{2T}.
\end{equation}

We prove this by splitting into the two cases $\{Z\ge1/2\}$ and $\{Z<1/2\}$.
On $\{Z\ge1/2\}$, the integrand is at most $\sqrt{2(1-Z)}$,
whose expectation is at most $\sqrt{2T}$ by Cauchy--Schwarz.

Now we deal with $\{Z<1/2\}$. Write $q=\Prob[Z<1/2]$ and note $q\le2T$. For $Q(u) = \inf\{z \,:\, \mathbb{P}[Z \leq z] \geq u \}$ with $u \in (0,1)$ note that stochastic domination implies $Q(u) \geq u$. Moreover, $Q(U)$ has the same distribution as $Z$ when $U$ is uniform on $(0,1)$. Therefore
\[
 \E \left[\ind_{\{Z<1/2\}} \sqrt{\frac{1-Z}{Z}} \right] \le \E[Z^{-1/2}\ind_{\{Z<1/2\}}]
 \le\int_0^q u^{-1/2}\,du=2\sqrt q\le2\sqrt{2T}.
\] 
The first inequality is trivial and the second inequality holds since $Q(U)$ has the same distribution as $Z$ and therefore $\E[Z^{-1/2}\ind_{\{Z<1/2\}}] = \int_0^1 Q(u)^{-1/2}1_{\{Q(u)<1/2\}} du$ and then apply $Q(u) \geq u$ and $1_{\{Q(u) < 1/2\}} \leq 1_{\{u \leq q\}}$.

We now write $F_X(a)=\Prob[X\le a]$ and
$T_X(u)=\Prob[X'>X+u]$ for a fixed $u>0$ and $X'$ an independent copy of $X$.
Note that $F_X(X)$ stochastically dominates a uniform variable and therefore the same
holds for the larger variable $F_X(X+u)$. Its expected complement
is $\Prob(X'>X+u)=T_X(u)$, and thus by \eqref{eq:quantile},
\begin{equation}\label{eq:tail-odds}
 \E\left[\sqrt{\frac{1-F_X(X+u)}{F_X(X+u)}}\right]
 \le3\sqrt{2T_X(u)}.
\end{equation}
The reflected estimate follows from
$1-F_Y(Y-u)\ge\Prob(Y'\ge Y\mid Y)$ for $u>0$; the latter variable
is the distribution function of $-Y$ evaluated at $-Y$ and therefore
also dominates uniform. The expected complement of $1-F_Y(Y-u)$ is
$\Prob(Y'\le Y-u)=T_Y(u)$ outside the finite set of positive
differences between atoms. Thus the analogous estimate holds for
$F_Y(Y-u)/(1-F_Y(Y-u))$ for almost every $u>0$.

Let $p_s(a) := \mathbb{P}[X\le a | X+Y=s]$ and let
\[
 q_s(a):=\frac{F_X(a)(1-F_Y(s-a))}
 {F_X(a)(1-F_Y(s-a))+(1-F_X(a))F_Y(s-a)}.
\]

For almost every $a$ we have
\begin{align*}
 H(\ind_{\{X\le a\}}\mid S) & = \E\bigl[-\ind_{\{X\le a\}}\log p_S(a)
             -\ind_{\{X>a\}}\log(1-p_S(a))\bigr] \\
             & = \E\bigl[-\ind_{\{X\le a\}}\log q_S(a)
             -\ind_{\{X>a\}}\log(1-q_S(a))\bigr] \\ & \quad\quad\quad\quad - \mathbb{E}[\mathcal{D}_{KL} (p_S(a) || q_S(a))]\\
 &\le\E\bigl[-\ind_{\{X\le a\}}\log q_S(a)
             -\ind_{\{X>a\}}\log(1-q_S(a))\bigr]
\end{align*}
where $$\mathcal{D}_{KL}(p || q) := p \log \frac{p}{q} + (1-p) \log \frac{1-p}{1-q} \geq 0.$$
Write $$L_a(X, S) := -\ind_{\{X\le a\}}\log q_S(a) -\ind_{\{X>a\}}\log(1-q_S(a)).$$
By the above and Tonelli's theorem we have
$$
\CPE(X\mid X+Y) \le \mathbb{E} \left[ \int_{\R} L_{X+u}(X, X+Y) \, du \right] .
$$

We therefore wish to estimate the integrals
$$
\int_{\R} L_a(x, x+y) \, da.
$$ On $a>x$, write $a=x+u$ with $u>0$. We have
\[
 L_{x+u}(x, x+y) = \log\left(1+
 \frac{1-F_X(x+u)}{F_X(x+u)}
 \frac{F_Y(y-u)}{1-F_Y(y-u)}\right).
\]
Using $\log(1+z)\le\sqrt z$, independence, and
\eqref{eq:tail-odds}, its expectation $\mathbb{E}[L_{X+u}(X, X+Y)]$ is at most
\[
 18\sqrt{T_X(u)T_Y(u)}\le9(T_X(u)+T_Y(u)).
\]
Similarly on $a<x$, write $a=x-u$. We have
\[
 L_{x-u}(x, x+y) = \log\left(1+
 \frac{F_X(x-u)}{1-F_X(x-u)}
 \frac{1-F_Y(y+u)}{F_Y(y+u)}\right).
\]
The reflected estimates give the same bound. Since
\[
 \int_0^\infty T_X(u)\,du=\tfrac12\Disp(X),
\]
integration over both halves proves \eqref{eq:cpe}.
\end{proof}

Before proving Lemma \ref{lem:dispersion} we need the following estimate for variance in terms of $\Disp$. Whilst it is straightforward to find a linear upper bound for the variance of the measure restricted to intervals of a given scale in terms of $\Disp$, it is necessary to look at many scales simultaneously in order to find a linear lower bound for the variance in terms of $\Disp$.

\begin{lemma} \label{lem:d_to_v}
      Suppose that $Y$ is a random variable. Define
      \[
            v_Y(t) := \E [ \Var (Y | P_t(Y)) ].\]
      Then $$\mathcal{D}(Y) \leq 2 \int_{0}^{\infty} v_Y(t) \, dt.$$
\end{lemma}

\begin{proof}
      Write $\rho = \mathcal{L}(Y)$ and let $Y'$ be an independent copy of $Y$ also independent of $P_t$. Then we have $$\mathbb{E}[(Y-Y')^2 \ind_{P_t(Y)=P_t(Y')} | P_t] = 2 \sum_{C \in P_t} \rho(C)^2 \Var (\rho | C) \leq  2 \sum_{C \in P_t} \rho(C) \Var (\rho | C).$$
      Two points at distance $d$ lie in the same cell of $P_t$ with probability $e^{-dt}$. Therefore averaging over $P_t$ we get $$ \E [(Y-Y')^2 e^{- t |Y-Y'|}] \leq 2 v_Y(t).$$
      Integrating over $t$ and applying Tonelli's theorem gives the required inequality.
\end{proof}

Next we prove Lemma \ref{lem:dispersion}.

\begin{lemma}[Poisson dispersion]\label{lem:dispersion}
For every finitely supported real random variable $Y$,
\begin{align}\label{eq:poisson-energy}
 \int_0^\infty\E_{P}[\Disp(Y\mid P_t(Y))]\,dt
 &\le2\int_0^\infty t\E_{P}[\Var(Y\mid  P_t(Y))]\,dt\notag\\
 &=3\int_0^\infty V_Y(r)\,\frac{dr}{r}=3\Energy(Y).
\end{align}
\end{lemma}

\begin{proof}

Let $P_t'$ be an independent process of Poisson cuts. By Lemma \ref{lem:d_to_v} we have
$$ \mathbb{E}_P[\Disp (Y | P_t(Y))] \leq 2 \int_{s=0}^{\infty} \E_{P, P'} [\Var (Y | P_t(Y), P_s'(Y)) ]\,ds.$$
Since the refinement of two independent Poisson partitions with intensities $t$ and $s$ is a Poisson partition with intensity $t+s$ the right hand side is $2 \int_{s=0}^{\infty} v_{Y}(s+t) \, ds$. Therefore
\begin{align*}
      \int_0^\infty\E_{P}[\Disp(Y\mid P_t(Y))]\,dt & \leq 2 \int_0^\infty t\E_{P}[\Var(Y\mid  P_t(Y))]\,dt.
\end{align*}

At intensity $t$, cells with left endpoint $a$
and length $r$ have intensity $t^2e^{-tr}\,da\,dr$. This is because left endpoints
have intensity $t\,da$, and the distance to the next cut is exponential
of rate $t$. Tonelli's theorem and
$\int_0^\infty t^3e^{-tr}\,dt=6r^{-4}$ therefore give
\begin{align*}
 \int_0^\infty t\E_{P}\Var(Y\mid P_t(Y))\,dt
 &=\int_0^\infty\int_0^\infty\int_\R
       t^3e^{-tr}F_\rho(I(a,r))\,da\,dr\,dt\\
 &=\frac32\int_0^\infty V_\rho(r)\,\frac{dr}{r}.
\end{align*}
This proves
\eqref{eq:poisson-energy}.
\end{proof}

We finish this section by proving Theorem~\ref{prop:addition}.

\begin{proof}[Proof of Theorem~\ref{prop:addition}]
Use the notation $X$, $S$, $P_t$ and $J(t)$ introduced at the
beginning of this section. We have $J(0)=0$. The partitions
eventually distinguish all support points almost surely, so bounded
convergence gives
\[
 \lim_{t\to\infty}J(t)
 =H(X\mid S)=\sum_{i=1}^m H(X_i)-H(S).
\]

To justify \eqref{eq:poisson-derivative}, restrict cuts to a compact
interval $K$ containing all supports. This does not change the
information carried by the cell labels. The restricted cuts form
a Poisson process $\eta_t$ on $E=\{1,\ldots,m\}\times K$ with
intensity $t\nu$, where
$\nu=\sum_i\delta_i\otimes(da|_K)$ is finite.
For a fixed configuration $\eta$, let $L_\eta(X)$ denote its tuple
of cell labels and set $f(\eta)=H(L_\eta(X)\mid S)$.
Then $f$ is measurable, $0\le f\le H(X)$, and
$J(t)=\E[f(\eta_t)]$. Thus \cite{LP17}*{Theorem 19.1}, applied
with base intensity $\lambda=0$, perturbation measure $\nu$ and
$n=1$, gives continuous differentiability and
\[
 J'(t)=\int_E
 \E\bigl[f(\eta_t+\delta_z)-f(\eta_t)\bigr]\,\nu(dz).
\]
At $t=0$ the derivative is understood from the right.
For $z=(i,a)$, adding the cut reveals the threshold test
$\ind_{\{X_i\le a\}}$, so the entropy chain rule gives, for almost
every $a$,
\[
 f(\eta+\delta_{(i,a)})-f(\eta)
 =H(\ind_{\{X_i\le a\}}\mid S,L_\eta(X)).
\]
Expanding over the cell labels and integrating proves
\eqref{eq:poisson-derivative}, since the threshold test is constant
on $X\in C$ whenever $a\notin C_i$.

Fix a product cell $C=C_1\times\cdots\times C_m$ of positive
probability. Conditional on $X\in C$, the coordinates remain
independent. Applying Lemma~\ref{lem:cpe} to $X_i$ and
$\sum_{j\ne i}X_j$ under this conditional law, and using the
triangle inequality for independent copies, gives
\[
 \int_{C_i}
 H(\ind_{\{X_i\le a\}}\mid S,X\in C)\,da
 \le 9\sum_{j=1}^m\Disp(X_j\mid X_j\in C_j).
\]
Substitute this into \eqref{eq:poisson-derivative}, sum over $i$
and average over the cells. The marginal cell weights give
\[
 J'(t)\le9m\sum_{i=1}^m
 \E_{P_t}\left[\Disp(X_i\mid P_t^{(i)}(X_i))\right].
\]
Integrating and applying Lemma~\ref{lem:dispersion}, we conclude
\[
 \sum_{i=1}^m H(X_i)-H(S)
 =\int_0^\infty J'(t)\,dt
 \le27m\sum_{i=1}^m\Energy(X_i),
\]
which proves the theorem since $27\le30$.
\end{proof}

\section{Proof of the main theorem}\label{sec:proof_main}

Throughout this section, as in the introduction, let $\mu=\sum_{i=1}^k p_i\delta_{g_i}$ be supported
on distinct contracting similarities $g_i(x)=r_ix+b_i$ of $\R$, with
$0<|r_i|<1$, $b_i\in \mathbb{R}$ and $p_i>0$. For a similarity $g(x)=rx+b$, write
$\rho(g)=r$ and $b(g)=b$ and denote $\rho_{\min} = \min_{1\leq i \leq k} |r_i|$. Let $X_1,X_2,\ldots$ be independent with
law $\mu$ and write
\[
G_n=X_1\cdots X_n,\qquad R_n=\rho(G_n),\qquad B_n=b(G_n).
\]
We use the same notation at stopping times for the natural filtration
of $(X_n)$.

An important fact  used throughout this section, as proved in \cite{Kfur}*{Lemma 2.19}, is that if $\tau$ is a
bounded stopping time and $Z'$ has law $\nu$ independently of the
stopped walk, then
\begin{equation}\label{eq:stationarity}
 B_\tau+R_\tau Z'\ \quad \text{has law }\nu.
\end{equation}

We now give an outline of the proof of Theorem~\ref{thm:main} and consider for simplicity the homogeneous case \(g_i(x)=\lambda x+b_i\) with \(\lambda \in (0,1)\). It is assumed for a contradiction that $\dim \nu < \min\{1,\frac{h_{\mu}}{|\chi_{\mu}|} \}$ and we aim to show that $\dim \nu = 1$. Recall that we discussed in the introduction how Theorem~\ref{prop:addition} together with \cite{H}*{Theorem 1.4} implies that for \(B_n=\sum_{\ell=0}^{n-1}\lambda^\ell D_\ell\) and $C > \lambda^{-1}$ that there is $c > 0$ such that for sufficiently large $n$
\[
\mathcal E_{C^{-n}}(B_n)
=\int_0^{C^{-n}}V_{B_n}(s)\,\frac{ds}{s}
\ge cn.
\]
Denote \(d=(n+1)\log\lambda^{-1}\). For \(u\in[0,d)\), consider the grid of scales \(e^{u+kd}\) with \(k\in\mathbb Z\). Averaging over \(u\) gives
\[
\frac1d\int_0^d
\sum_{\substack{k\in\mathbb Z\\e^{u+kd}<C^{-n}}}
V_{B_n}(e^{u+kd})\,du
=\frac1d\int_0^{C^{-n}}V_{B_n}(s)\,\frac{ds}{s}
\ge\frac{cn}{d}.
\]
Thus some grid has total variance contribution at least \(cn/d\). Since \(V_{B_n}(s)=0\) for sufficiently small \(s\), retaining the nonzero contributions gives finitely many scales \(s_1<\cdots<s_m<C^{-n}\) satisfying
\[
s_{j+1}\ge\lambda^{-(n+1)}s_j
\quad \text{ and } \quad 
\sum_{j=1}^m V_{B_n}(s_j)
\ge\frac{cn}{(n+1)\log\lambda^{-1}}
\ge\frac{c}{2\log\lambda^{-1}}.
\]
Each \(V_{B_n}(s_j)\) is the average conditional variance obtained by conditioning \(B_n\) on its cell in a randomly shifted grid of mesh \(s_j\), multiplied by \(4/s_j^2\).

To bring these contributions to a common scale \(r\le s_1\), choose integers \(k_j\ge0\) satisfying
\[
\lambda\frac r{s_j}<\lambda^{k_j}\le\frac r{s_j}.
\]
The separation of the scales ensures that the digit blocks \([k_j,k_j+n)\) are disjoint. Thus the random-series representation of \(\nu\) contains the independent summands
\[
S=\sum_{j=1}^m\lambda^{k_j}B_n^{(j)},
\]
where the \(B_n^{(j)}\) are independent copies of \(B_n\). $S$ is therefore part of a sample of $\nu$ and the unused digits form an independent remainder.

Condition each \(B_n^{(j)}\) on its cell in an independently shifted grid of mesh \(s_j\), and let \(\mathcal F\) record the shifts and cell labels. This fixes the coarse location of each block, leaving its fluctuation inside an interval of length \(s_j\). The blocks remain conditionally independent. After centering and rescaling, the variables
\[
X_j=\lambda^{k_j}\bigl(B_n^{(j)}
-\mathbb E[B_n^{(j)}\mid\mathcal F]\bigr)
\]
satisfy \(|X_j|\le\lambda^{k_j}s_j\le r\). Thus conditioning turns the local variance into the variance of small independent summands. Moreover,
\[
\begin{aligned}
\frac4{r^2}\mathbb E[\operatorname{Var}(S\mid\mathcal F)] &= \frac{4}{r^2} \sum_{j = 1}^m \mathbb{E}[\operatorname{Var}(X_j\mid\mathcal F)] \\
&=\sum_{j=1}^m\frac{\lambda^{2k_j}s_j^2}{r^2}\,V_{B_n}(s_j)\ge\lambda^2\sum_{j=1}^mV_{B_n}(s_j)
\ge\frac{c\lambda^2}{2\log\lambda^{-1}}.
\end{aligned}
\]
No mass is discarded as the original law is a mixture of these conditional laws and centering only translates each component.

For a fixed conditional component, denote \(\sigma^2=\sum_j\operatorname{Var}(X_j\mid\mathcal F)\). Since the \(X_j\) are centered, conditionally independent and satisfy \(|X_j|\le r\), a suitable version of the Berry–Esseen theorem approximates \(\sigma^{-1}\sum_jX_j\) by a standard Gaussian with Wasserstein error bounded by
\[
C\frac{\sum_j\mathbb E[|X_j|^3\mid\mathcal F]}{\sigma^3}
\le C\frac r\sigma,
\]
for $C$ an absolute constant. Thus a good approximation requires the total variance \(\sigma^2\) to be large compared with \(r^2\). Taking \(r\to0\) alone does not guarantee this, since \(\sigma\) may decrease at the same rate.

The key point allowing us to repeat the construction is that \(C>\lambda^{-1}\). Indeed, if \(a>0\) is the smallest scale already selected, we can choose a new block length \(N\) sufficiently large satisfying \(C^{-N}<\lambda^{N+1}a\). Applying the energy estimate to \(B_N\) then supplies a new group of scales below \(C^{-N}\), with the same uniform positive lower bound on their total variance contribution. This separates the new scales sufficiently from the previous ones that, upon placing all groups at a sufficiently small common scale \(r\), their associated digit blocks remain disjoint. Iterating gives arbitrarily many independent groups, each contributing a fixed positive amount to the expected variance normalized by \(r^2\). Using the latter, we will conclude $\dim \nu = 1$, which is a contradiction and therefore shows that $\dim \nu = \min\{1,\frac{h_{\mu}}{|\chi_{\mu}|} \}$.

To make these arguments also work in the non-homogeneous case, we work with conditioning on the rate of contraction at step $n$. To do so, we use conditional laws. Indeed, for random variables $Y$ and $T$, we write
$\cL(Y\mid T)$ for its conditional law, that is, the a.s.-unique
$\sigma(T)$-measurable random probability measure satisfying
\[
 \E[\mathbb{I}_E\cL(Y\mid T)(A)]=\Prob(Y\in A,E)
\]
for every Borel set $A\subset\R$ and every $E\in\sigma(T)$.

In Section~\ref{sec:entlocvar}, we use Theorem~\ref{prop:addition} and \cite{H}*{Theorem 1.4} to deduce the suitably conditional version of $\mathcal E_{C^{-n}}(B_n) \geq cn$ as outlined in the introduction. We highlight that Section~\ref{sec:entlocvar} is the core novelty of the proof in the deduction of Theorem~\ref{thm:main} from Theorem~\ref{prop:addition}. In Section~\ref{sec:Wdef} we introduce the quantity $W(\lambda; r)$ tracking the amount of variance at scale $r$ the measure $\lambda$ has and show that $W(\lambda; r) \to 0$ as $r\to 0$ implies that $\dim \lambda = 1$. The simplified version of the variance summation method of \cite{KK25} is discussed in Section~\ref{sec:VarianceSumm}. Finally, we conclude the proof of Theorem~\ref{thm:main} in Section~\ref{sec:conclusion}.

\subsection{Entropy and local variance}\label{sec:entlocvar}
We use the same notion of entropy at a scale as
Varj\'u \cite{Vabs}*{Sections~1.2 and~2.2}, which measures entropy by averaging over
translations. The concept originates in Wang \cite{Wan11}*{Section~4.1}.
For a integrable real-valued random variable $Y$, set
\begin{equation}\label{eq:scale-entropy}
 H(Y;r)=\frac1r\int_0^r
 H\left(\left\lfloor\frac{Y+t}{r}\right\rfloor\right)\,dt.
\end{equation}
This is concave in the law of $Y$. As will be important later, if $Y$ is finitely supported and
$Z_t=\lfloor(Y+t)/r\rfloor$, the chain rule for entropy gives $H(Y) = H(Y,Z_t) = H(Z_t) + H(Y|Z_t)$ and therefore
\begin{equation}\label{eq:conditional-entropy}
 H(Y)-H(Y;r)=\frac1r\int_0^r H(Y\mid Z_t)\,dt.
\end{equation} We also use Varj\'u's notion of entropy between two scales for some $C > 0$,
\[
 H(Y;r\mid Cr):=H(Y;r)-H(Y;Cr).
\]The following is stated in
\cite{Vabs}.
\begin{lemma}(\cite{Vabs}*{Lemma 10})\label{lemma:entropy_concave}
Suppose that $C>1$ is an integer. Then $H(Y;r\mid Cr)$ is concave
in the law of $Y$.
\end{lemma}

We use the following estimate from Hochman.
\begin{lemma}[Hochman's entropy estimate]\label{lem:hochman}
If $\dim \nu<1$, then, for every $C>e^{|\chi_{\mu}|}$,
\[
 H(B_n;C^{-n})=n|\chi_{\mu}| \dim \nu+o(n).
\]
\end{lemma}
\begin{proof}
Let $n'=n|\chi_\mu|/\log 2$. Hochman's theorem
\cite{H}*{Theorem~1.4 and the discussion following it}, after converting
to natural logarithms, gives
\[
 H\bigl(\lfloor2^{qn'}B_n\rfloor,R_n\bigr)
 =n|\chi_\mu|\dim\nu+o(n)\qquad(q>1).
\]
Integer parts of the partition levels change the entropy by $O(1)$.
Clearly $H(R_n)\le k\log(n+1)=o(n)$. We may therefore discard $R_n$.
Any two translated grids of the same mesh have entropies differing
by at most $\log 2$, since each cell of either grid meets at most two
cells of the other. Averaging over the translation and taking
$q=\log C/|\chi_\mu|$ proves the claim.
\end{proof}

We next use Theorem~\ref{prop:addition} to prove a general inequality for finitely supported random variables.

\begin{proposition}\label{prop:finite-sum}
Suppose $Y=\sum_{j=1}^L a_jD_j$, where the $a_j$ are real constants
and the $D_j$ are independent real random variables, each taking at
most $K$ values. For every $R>0$ and positive integer $m$,
\begin{equation}\label{eq:finite-sum}
 \cE_R(Y)\ge\frac1{30m}\left(
 H(Y)-H(Y;R)-\frac{L(K-1)\log(m+1)}m\right)-\frac12.
\end{equation}
\end{proposition}
\begin{proof}
Take independent copies of the digit vector and put
$Y_i=\sum_{j=1}^L a_jD_j^{(i)}$ and $S=\sum_{i=1}^mY_i$.
For each $j$, the sum $\sum_iD_j^{(i)}$ is determined by the
multiplicities of its possible values. The first $K-1$ multiplicities
determine the last, so
\[
 H(S)\le L(K-1)\log(m+1).
\]
Fix $t\in[0,R)$ and write $Q_i=\lfloor(Y_i+t)/R\rfloor$ and
$Q=\lfloor(Y+t)/R\rfloor$. Conditioned on $(Q_1,\ldots,Q_m)$,
the $Y_i$ remain independent. Apply Theorem~\ref{prop:addition}, average over
the $Q_i$, and use $H(S\mid Q_1,\ldots,Q_m)\le H(S)$ to obtain
\[
 mH(Y\mid Q)\le H(S)+30m^2\E\bigl[\cE(\cL(Y\mid Q))\bigr].
\]
Each conditional law is supported on an interval of length $R$.
The tail bound and concavity from Lemma~\ref{lem:variance-basic} with $D = R$ therefore give
\[
 \E\bigl[\cE(\cL(Y\mid Q))\bigr]
 \le\E\bigl[\cE_R(\cL(Y\mid Q))\bigr]+\frac12
 \le\cE_R(Y)+\frac12.
\]
Average over $t$, use \eqref{eq:conditional-entropy}, and divide by
$m$. Rearranging gives \eqref{eq:finite-sum}.
\end{proof}

We now convert Lemma~\ref{lem:hochman} and Proposition~\ref{prop:finite-sum} to an estimate of the conditional variance energy after conditioning on $R_n$.

\begin{corollary}\label{cor:fine-variance}
Suppose that $\dim \nu<\min\{1,h_\mu/|\chi_{\mu}|\}$. Then there is $\eta>0$, depending
only on $\mu$, such that, for every $C>e^{|\chi_{\mu}|}$ and all sufficiently
large (in terms of $C$) $n$,
\begin{equation}\label{eq:fine-variance}
 \E\bigl[\cE_{C^{-n}}(\cL(B_n\mid R_n))\bigr]\ge\eta n.
\end{equation}
\end{corollary}
\begin{proof}
Put $\delta=h_\mu-|\chi_{\mu}| \dim \nu >0$. Choose a block length $N$ such that
$k\log(N+1)/N<\delta/4$. To reveal only little entropy, split the first $n$ increments into
$L=\lceil n/N\rceil$ consecutive blocks of length $N$, except
possibly the last, and write their products as $x\mapsto A_jx+D_j$.
Set $T=(A_1,\ldots,A_L)$. Each $A_j$ takes at most $(N+1)^k$ values,
and $R_n$ is determined by $T$. Hence
\begin{align*}
 H(B_n\mid T)
 &=H(G_n\mid T)\ge H(G_n)-H(T)\\
 &\ge nh_\mu-Lk\log(N+1)
 \ge n(h_\mu-\delta/4)-O(1),
\end{align*}
where $H(G_n)\ge nh_\mu$ follows from subadditivity of random-walk
entropy. Concavity of entropy at a scale and Lemma~\ref{lem:hochman}
give
\[
 \E[H(\cL(B_n\mid T);C^{-n})]
 \le H(B_n;C^{-n})=n|\chi_{\mu}| \dim \nu +o(n).
\]
Thus the average conditional entropy gap is at least $\delta n/2$
for large $n$.

Given $T$, the $D_j$ remain independent and
\[
 B_n=D_1+A_1D_2+\cdots+(A_1\cdots A_{L-1})D_L
\]
has fixed coefficients. Each digit has at most $K=k^N$ values.
Choose $m$ so large that
$(K-1)\log(m+1)/m<\delta N/8$.
Apply Proposition~\ref{prop:finite-sum} conditionally and average.
For large $n$, this gives
\[
 \E\bigl[\cE_{C^{-n}}(\cL(B_n\mid T))\bigr]
 \ge\frac{\delta n}{120m}-\frac12.
\]
Finally, $R_n$ is a function of $T$, so concavity of $\cE_{C^{-n}}$
allows us to replace $T$ by $R_n$. Decreasing the constant proves
\eqref{eq:fine-variance}, for example with $\eta=\delta/(240m)$.
\end{proof}

\subsection{Convolution disintegrations}\label{sec:Wdef}
We introduce the following functional to record how much variance can be placed
in independent summands at a common scale.

\begin{definition}\label{def:convolution-variance}
For a compactly supported probability measure $\lambda$ and $r>0$,
a scale-$r$ convolution disintegration is a representation
\[
 \lambda=\int_I \lambda_i^{(1)}*\cdots*
 \lambda_i^{(m_i)}\,d\theta(i),
\]
where $(I,\theta)$ is a probability space, $m_i$ is a finite positive
integer, and each factor is supported on an interval of length at
most $r$. All families are required to be measurable. Define
\begin{equation}\label{eq:convolution-variance}
 W(\lambda;r)=\inf\int_I
 \exp\left(-\frac4{r^2}\sum_{j=1}^{m_i}
 \Var(\lambda_i^{(j)})\right)\,d\theta(i),
\end{equation}
with the infimum over these disintegrations.
\end{definition}

We use $W$ as it has many suitable properties, as shown in Lemma \ref{lem:convolution-properties} and in particular it is well behaved under convolution. We now provide some further motivation for this definition. At every sufficiently small scale, we aim to decompose the self-similar measure into components that are approximately normal, together with a small remainder. The standard deviations of these normal distributions should be comparable to the given scale. Standard Berry--Esseen estimates (Lemma~\ref{lemma:approx_normal}) give the required normal approximations. Lemma~\ref{lemma:wasserstein_enough} then provides entropy bounds between two scales for measures close to normal distributions. By concavity of entropy between two scales, these bounds imply that the self-similar measure is full dimensional.

We use $W(\cdot; \cdot)$ to keep track of this approximation. Suppose that $W(\nu; r)$ is sufficiently small. A simple application of Markov's inequality then gives the desired decomposition after discarding an arbitrarily small fraction of the mass of $\nu$. Each remaining component is the law of a sum of many independent random variables. These laws approximate normal distributions with variance $Cr^2$ arbitrarily well. More precisely, we prove the following proposition.

\begin{proposition}\label{prop:w_to_zero}
Suppose that $\nu$ is a self-similar measure on $\R$ and that there is
a sequence $r_n\to0$ such that $W(\nu;r_n)\to0$. Then $\dim\nu=1$.
\end{proposition}

In \cite{KK25} we do not use $W(\cdot;r)$ but instead use effective Cramér-type bounds to achieve the same goal of showing that all but an arbitrarily small portion of the measure can be approximated by the law of the sum of many independent random variables. This has the advantage of being more efficient but leads to a more complicated argument. However, both approaches are closely related.

The remaining aim of this section is to prove Proposition~\ref{prop:w_to_zero}, which we establish by showing that $H(\nu; r) = \log r^{-1} + o(\log r^{-1})$ as $r\to 0$, which implies $\dim\nu = 1$ by exact dimensionality. We first prove some basic properties of $W(\cdot;r)$. We will then state and prove the lemmas necessary to prove Proposition \ref{prop:w_to_zero} and prove Proposition \ref{prop:w_to_zero}.

\begin{lemma}\label{lem:convolution-properties}
The functional $W$ takes values in $[0,1]$ and is convex in its
measure argument. Moreover,
\begin{align}
 W(aY+b;r)&=W(Y;r/|a|) &&(a\ne0),\label{eq:affine-W}\\
 W(\lambda*\theta;r)&\le W(\lambda;r)W(\theta;r),\label{eq:product-W}\\
 W(\lambda;s)&\le W(\lambda;r)^{r^2/s^2} &&(s\ge r),\label{eq:scale-W}\\
 W(\lambda;r)&\le1-(1-e^{-1})V_\lambda(r).\label{eq:local-to-convolution}
\end{align}
\end{lemma}

This lemma is important as we will use \eqref{eq:local-to-convolution} to estimate $W$ of the translation part of the product of $n$ similarities after conditioning on the linear part in terms of the variance energy. We then use the other properties to put together estimates for many such products into an improved estimate.

\begin{proof}[Proof of Lemma \ref{lem:convolution-properties}]
Clearly $W\le1$. Convexity follows by combining disintegrations, affine changes of variables give \eqref{eq:affine-W}, and
concatenating two independent disintegrations gives
\eqref{eq:product-W}. A scale-$r$ convolution disintegration is also a
scale $s$ convolution disintegration for $s\ge r$. From the definition it is immediate that if its integrand at scale $r$ is $Z\in[0,1]$,
its integrand at scale $s$ is $Z^{r^2/s^2}$. Jensen's inequality
followed by the infimum gives \eqref{eq:scale-W}.

For \eqref{eq:local-to-convolution}, condition $\lambda$ on a shifted
grid of mesh $r$, taking one factor in each component, and average
over a uniform shift. If $Q$ denotes the resulting cell label, every
normalized component variance
$u=4\Var(\lambda\mid Q)/r^2$ lies in $[0,1]$. The inequality
$e^{-u}\le1-(1-e^{-1})u$ and the definition of $V_\lambda(r)$
give the result.
\end{proof}

We need the following Lemma for the proof of Proposition~\ref{prop:w_to_zero}.

\begin{lemma}\label{lemma:approx_normal}
There is an absolute constant $C>0$ such that the following holds.
Suppose that $Y_1,\ldots,Y_n$ are independent real-valued random
variables with finite third moments and mean zero. Let
\[
 \omega_i^2=\E[Y_i^2],\qquad \gamma_i^3=\E[|Y_i|^3],
 \qquad r^2=\sum_{i=1}^n\omega_i^2>0.
\]
Then
\[
 \mathcal W_1\left(\sum_{i=1}^nY_i,\mathcal N(0,r^2)\right)
 \le C\frac{\sum_{i=1}^n\gamma_i^3}{r^2}.
\]
In particular, if $|Y_i|\le\varepsilon$ almost surely for every $i$,
then
\[
 \mathcal W_1\left(\sum_{i=1}^nY_i,\mathcal N(0,r^2)\right)
 \le C\varepsilon.
\]
\end{lemma}

\begin{proof}
    Many estimates of this form are available in the literature and we
    cite \cite{Gol10}*{Theorem~1.1} as one convenient version. In one
    dimension, the Wasserstein distance is the $L^1$ distance from \cite{Gol10}*{Theorem~1.1}. So \cite{Gol10}*{Theorem~1.1}, after
    rescaling by $r$, gives the first assertion with $C=1$. For the
    second assertion simply note that
    $\gamma_i^3\leq\varepsilon\omega_i^2$ and sum over $i$.
\end{proof}

The following lemma allows us to estimate entropy at a given scale from a normal approximation in Wasserstein distance. We will apply it at two scales to estimate the entropy between them.

\begin{lemma}\label{lemma:wasserstein_enough}
For every $0<a<b<\infty$ and $\varepsilon>0$, there is $\delta>0$
such that the following holds. If $r>0$, $\sigma\in[ar,br]$, and $Y$
is a mean-zero random variable satisfying
\[
 \E[Y^2]\le b^2r^2,
 \qquad
 \mathcal W_1(Y,\mathcal N(0,\sigma^2))<\delta r,
\]
then
\[
 \left|H(Y;r)-H(\mathcal N(0,\sigma^2);r)\right|<\varepsilon.
\]
\end{lemma}

\begin{proof}
By scaling, it suffices to take $r=1$. If the assertion failed, there
would be mean-zero random variables $Y_n$ and numbers
$\sigma_n\in[a,b]$ such that
\[
 \E[Y_n^2]\le b^2,\qquad
 \mathcal W_1(Y_n,\mathcal N(0,\sigma_n^2))\longrightarrow0,
\]
but the asserted entropy convergence fails. Passing to a subsequence,
we may assume that $\sigma_n\to\sigma\in[a,b]$. It follows that
$Y_n$ converges in distribution to $\mathcal N(0,\sigma^2)$.

For $k\in\Z$ and $t\in[0,1]$, write
\[
 A_{k,t}^{(n)}=\Prob[Y_n\in[k-t,k+1-t)],
 \quad \text{and} \quad
 B_{k,t}=\Prob[\mathcal N(0,\sigma^2)\in[k-t,k+1-t)].
\]
Then $A_{k,t}^{(n)}\to B_{k,t}$ for every $k$ and almost every $t$.
Moreover, Chebyshev's inequality gives, uniformly in $n$ and $t$, $A_{k,t}^{(n)}\ll_b(1+|k|)^{-2}.$ Consequently the terms $-A_{k,t}^{(n)}\log A_{k,t}^{(n)}$ admit a
summable dominating function. Dominated convergence gives $H(Y_n;1)\rightarrow H(\mathcal N(0,\sigma^2);1)$ as $n\to \infty$, which is
a contradiction.
\end{proof}

We need the following to estimate the entropy between two scales.

\begin{lemma}\label{lemma:variance_sum_to_entropy}
For every integer $C>1$ and $\varepsilon>0$, there are $\delta>0$
and $A>0$ such that the following holds. If $r>0$ and
$Y_1,\ldots,Y_n$ are independent mean-zero real-valued random
variables satisfying
\[
 |Y_i|\le\delta r\quad\text{almost surely},
 \qquad
 \sum_{i=1}^n\Var(Y_i)\ge Ar^2,
\]
then
\[
 H\left(\sum_{i=1}^nY_i;r\mid Cr\right)>\log C-\varepsilon.
\]
\end{lemma}

\begin{proof}
By scale invariance and the smoothness of the Gaussian density, we may
choose $A$ sufficiently large such that whenever $Ar^2\le\sigma^2\le(A+1)r^2$ it holds that
\[
 H(\mathcal N(0,\sigma^2);r\mid Cr)>
 \log C-\frac{\varepsilon}{2}.
\]
Choose $\delta<1$ sufficiently
small for Lemmas~\ref{lemma:approx_normal} and
\ref{lemma:wasserstein_enough}, applied at both scales $r$ and $Cr$,
to give a total entropy error of less than $\varepsilon/2$.

Take the shortest initial subsum $S=Y_1+\cdots+Y_\ell$ whose variance
is at least $Ar^2$. Writing $\sigma^2=\Var(S)$, minimality and
$\Var(Y_i)\le\delta^2r^2$ give
\[
 Ar^2\le\sigma^2\le(A+\delta^2)r^2\le(A+1)r^2.
\]
Lemma~\ref{lemma:approx_normal} gives
\[
 \mathcal W_1(S,\mathcal N(0,\sigma^2))\ll\delta r.
\]
It follows from Lemma~\ref{lemma:wasserstein_enough}, used at the two
scales, that
\[
 H(S;r\mid Cr)>\log C-\varepsilon.
\]
Finally, adding the remaining independent summands cannot decrease
entropy between these scales by, for example, Lemma~\ref{lemma:entropy_concave} and the translation invariance of entropy between two scales.
\end{proof}

We now prove Proposition \ref{prop:w_to_zero}.

\begin{proof}[Proof of Proposition \ref{prop:w_to_zero}]
    We first show $W(\nu;s)\to0$ as $s\to0$. For $0<s\le t$,
let $\sigma=\min\{n\ge0:|R_n|\le s/t\}$. This stopping time is
bounded, and for $\rho_{\min} = \min_{1\leq i \leq k} |r_i|$ we have
\[
 \rho_{\min} s/t<|R_\sigma|\le s/t.
\]
By \eqref{eq:stationarity}, convexity, affine invariance and
\eqref{eq:scale-W},
\begin{equation}\label{eq:propagate}
 W(\nu;s)\le\E[W(\nu;s/|R_\sigma|)]
 \le W(\nu;t)^{\rho_{\min}^2}.
\end{equation}
Passing to a decreasing subsequence of $(r_n)$ and applying
\eqref{eq:propagate} with $t=r_n$ shows that $W(\nu;s)\to0$ as
$s\to0$.

Next we show that $H(\nu;r\mid2r)\to\log2$ as $r\to0$. Let
$\varepsilon>0$, and choose $\delta$ and $A$ from
Lemma~\ref{lemma:variance_sum_to_entropy}, with $C=2$. Since
$W(\nu;\delta r)\to0$, for all sufficiently small $r$ there is a
scale-$\delta r$ convolution disintegration
\[
 \nu=\int_I\lambda_i^{(1)}*\cdots*\lambda_i^{(m_i)}\,d\theta(i)
\]
such that, writing
$S_i=\sum_{j=1}^{m_i}\Var(\lambda_i^{(j)})$,
\[
 \int_I\exp\left(-\frac{4S_i}{\delta^2r^2}\right)d\theta(i)
 <\varepsilon\exp\left(-\frac{4A}{\delta^2}\right).
\]
It follows that $\theta(\{i:S_i<Ar^2\})<\varepsilon$. For every other $i$, Lemma~\ref{lemma:variance_sum_to_entropy} gives
\[
 H\left(\lambda_i^{(1)}*\cdots*\lambda_i^{(m_i)};
 r\mid2r\right)\ge\log2-\varepsilon.
\]
The concavity of entropy between two scales therefore gives
\[
 H(\nu;r\mid2r)\ge\log2-O(\varepsilon).
\]
Consequently, as $r \to 0$,
\[
 H(\nu;r)=\log r^{-1}+o(\log r^{-1}),
\]
and hence $\dim\nu=1$.
\end{proof}

\subsection{Variance Summation}\label{sec:VarianceSumm}
Throughout this section write $\rho_{\min} = \min_{1\leq i \leq k} |r_i|$. The following lemma allows us to combine variance found in different
stopped random walks and is a simple one-dimensional analogue of \cite{KKabs}*{Proposition 6.4}. As outlined in the sketch of the argument, we will apply Lemma~\ref{lem:concatenation} twice, namely in the proof of Proposition~\ref{prop:stopped-block} and of Theorem~\ref{thm:main}.

\begin{lemma}[Concatenation]\label{lem:concatenation}
Let $0<s_1<\cdots<s_m$ and let $\tau_1,\ldots,\tau_m$ be bounded
stopping times for independent walks with independent increments of law
$\mu$.
Write $(C_j,D_j)$ for their signed contraction ratio and translation.
Suppose that, for $j<m$,
\begin{equation}\label{eq:separation}
 s_j\le\rho_{\min} s_{j+1}|C_j|\quad\text{almost surely}.
\end{equation}
For $0<r\le s_1$, these blocks can be concatenated to one walk ending
at a bounded stopping time $T$, such that
\begin{equation}
 \E[W(\cL(B_T\mid R_T);r)]
 \le\prod_{j=1}^m
 \E[W(\cL(D_j\mid C_j);s_j)]^{\rho_{\min}^2},
 \label{eq:concatenation-prefix}
 \end{equation}
 and consequently
 \begin{equation}
     W(\nu;r)
 \le\prod_{j=1}^m
 \E[W(\cL(D_j\mid C_j);s_j)]^{\rho_{\min}^2}.
 \label{eq:concatenation-nu}
 \end{equation} Moreover, $|R_T|>\rho_{\min} r|C_m|/s_m$ almost surely.
\end{lemma}
\begin{proof}
For $1\leq j \leq m$ define
\[
 S_j=\min\{n\ge0:|R_n|\le r/s_j\} \quad\text{ and } \quad
T_j=S_j+\tau_j^{(S_j)},
\]
where $\tau_j^{(S_j)}$ applies the $j$th stopping rule to the
increments after $S_j$. In other words the law of $(X_{S_j+1}, X_{S_j+2},\dots, X_{S_j+\tau_j^{(S_j)}})$ given $X_1, \dots, X_{S_j}$ is $\mu^{*\tau_j}$. Denote $A_j=R_{S_j}$. By construction
\begin{equation}\label{eq:overshoot}
 \rho_{\min} r/s_j<|A_j|\le r/s_j.
\end{equation}
Thus \eqref{eq:separation} implies
\[
 |R_{T_j}|=|A_jC_j|>r/s_{j+1},
\]
so $T_j<S_{j+1}$. Therefore, the selected blocks are disjoint and their
stopped products are independent. As $\max_i|r_i|<1$ all stopping times
are uniformly bounded.
We take $T=T_m$ and the final assertion that $|R_T|>\rho_{\min} r|C_m|/s_m$ follows from \eqref{eq:overshoot}.

Let $\cA$ include the ratios $C_1,\ldots,C_m$ and the complete words
in all intervening gaps including the initial gap. Then the $A_j$
and $R_T$ are $\cA$-measurable, and
\begin{equation}\label{eq:conditional-blocks}
 B_T=d_{\cA}+\sum_{j=1}^m A_jD_j
\end{equation}
for an $\cA$-measurable translation $d_{\cA}$. Given $\cA$, the $D_j$
are independent and have laws $\cL(D_j\mid C_j)$. 

Put $v_j(c)=W(\cL(D_j\mid C_j=c);s_j)$. From
\eqref{eq:conditional-blocks}, \eqref{eq:product-W} and
\eqref{eq:overshoot},
\[
 W(\cL(B_T\mid\cA);r)
 \le\prod_{j=1}^m
 W(\cL(D_j\mid C_j);r/|A_j|)
 \le\prod_{j=1}^m v_j(C_j)^{\rho_{\min}^2}.
\]
Average and use convexity to pass from $\cA$ to $R_T$. The $C_j$
are independent, so Jensen's inequality yields
\[
 \E\left[\prod_jv_j(C_j)^{\rho_{\min}^2}\right]
 =\prod_j\E[v_j(C_j)^{\rho_{\min}^2}]
 \le\prod_j \E[v_j(C_j)]^{\rho_{\min}^2}.
\]
This proves \eqref{eq:concatenation-prefix}. Adding the
tail $R_TZ'$ with $Z' \sim \nu$ independent of the random walk proves \eqref{eq:concatenation-nu}.
\end{proof}

We next state the following proposition, which is similar to \cite{KKabs}*{Proposition 6.7} or \cite{KK25}*{Proposition 5.8} but importantly does not rely on any separation assumption as we use Corollary~\ref{cor:fine-variance}.

\begin{proposition}\label{prop:stopped-block}
Suppose $\dim \nu<\min\{1,h_\mu/|\chi_{\mu}|\}$. Then there is $\kappa>0$ such that,
for every $\varepsilon>0$, there is a bounded stopping time $\tau$
and some $r>0$ satisfying
\begin{equation}\label{eq:stopped-block}
 \E[W(\cL(B_\tau\mid R_\tau);r)]\le e^{-\kappa},
 \qquad r<\varepsilon|R_\tau|\quad\text{almost surely}.
\end{equation}
\end{proposition}
\begin{proof}
Fix $C>\rho_{\min}^{-1}$ with $\rho_{\min} = \min_{1 \leq i \leq k} |r_i|$ and let $\eta$ be as in
Corollary~\ref{cor:fine-variance}. For sufficiently large $N$, write
\[
 w_N(s)=\E[V_{\cL(B_N\mid R_N)}(s)] \quad \text{ and } \quad
  d_N=(N+1)\log\rho_{\min}^{-1}.
\]
Then $\int_0^{C^{-N}}w_N(s)\,ds/s\ge\eta N$ for a suitable $\eta > 0$.
The conditional laws involved are finitely supported, so $w_N(s)=0$
for all sufficiently small $s$. Averaging the logarithmic grids
$\{e^{u+jd_N}:j\in\mathbb Z\}$ over $u\in[0,d_N)$ gives
\[
 \frac1{d_N}\int_0^{d_N}
 \sum_{e^{u+jd_N}<C^{-N}}w_N(e^{u+jd_N})\,du
 =\frac1{d_N}\int_0^{C^{-N}}w_N(s)\,\frac{ds}{s}
 \ge\frac{\eta}{2\log\rho_{\min}^{-1}}.
\]
Consequently, some finite collection
$0<s_1<\cdots<s_m<C^{-N}$ satisfies
\begin{equation}\label{eq:selected-scales}
 s_{j+1}\ge\rho_{\min}^{-(N+1)}s_j \quad \text{ and } \quad \sum_{j=1}^m w_N(s_j)\ge
 \eta_0:=\frac{\eta}{2\log\rho_{\min}^{-1}}>0.
\end{equation}
Apply Lemma~\ref{lem:concatenation} with $r=s_1$ and with every
block having the fixed length $N$. Condition \eqref{eq:separation}
holds because $|R_N|\ge\rho_{\min}^N$. Using
\eqref{eq:local-to-convolution} and $1-x\le e^{-x}$ gives
\begin{align*}
 \E[W(\cL(B_T\mid R_T);r)]
 &\le\prod_{j=1}^m
 \bigl(1-(1-e^{-1})w_N(s_j)\bigr)^{\rho_{\min}^2}\\
 &\le\exp\bigl(-\rho_{\min}^2(1-e^{-1})\eta_0\bigr).
\end{align*}
Moreover, the last assertion of Lemma~\ref{lem:concatenation} gives
\[
 \frac r{|R_T|}<\frac{s_m}{\rho_{\min}^{N+1}}
 <\rho_{\min}^{-1}(C\rho_{\min})^{-N}.
\]
Since $C\rho_{\min}>1$, choose $N$ sufficiently large that the last
expression is at most $\varepsilon$. Taking $\tau=T$ and
$\kappa=\rho_{\min}^2(1-e^{-1})\eta_0$ proves the proposition.
\end{proof}

\subsection{Conclusion}\label{sec:conclusion}
\begin{proof}[Proof of Theorem~\ref{thm:main}]
The upper bound
\[
 \dim \nu \leq \min \left\{1,\frac{h_\mu}{|\chi_\mu|}\right\}
\]
is well known. In particular it follows from Feng--Hu \cite{FH}. Indeed, apply
their projection-entropy bound \cite{FH}*{Theorem~2.2(i)} to the system given by the distinct
similarities in $\supp(\mu^{*n})$. Its symbolic entropy is
$H(\mu^{*n})$ and its Lyapunov exponent is $n\chi_\mu$, so
$\dim\nu\leq\min\{1,
 \frac{H(\mu^{*n})}{n|\chi_\mu|}\}$.
Letting $n\to\infty$ gives the claimed upper bound.
It remains to show that $\dim \nu \geq \min \{ 1, \frac{h_{\mu}}{|\chi_{\mu}|} \}$. Assume for contradiction that this does not hold.

Choose $(\tau_1,r_1)$ from
Proposition~\ref{prop:stopped-block} with $1$ in the role of $\varepsilon$.
Inductively choose $(\tau_{j+1},r_{j+1})$ with $\rho_{\min} r_j$
 in the role of $\varepsilon$. Thus
\begin{equation}\label{eq:recursive-scales}
 r_{j+1}<\rho_{\min} r_j|R_{\tau_{j+1}}|
 \quad \text{ and } \quad
 \E[W(\cL(B_{\tau_j}\mid R_{\tau_j});r_j)]\le e^{-\kappa}.
\end{equation} For each $n$, apply Lemma~\ref{lem:concatenation} at the scale
$r_n$ with the blocks ordered as
\[
 (\tau_n,r_n),(\tau_{n-1},r_{n-1}),\ldots,(\tau_1,r_1).
\]
The separation hypothesis is exactly \eqref{eq:recursive-scales}.
Consequently
\[
 W(\nu;r_n)\le e^{-n\kappa\rho_{\min}^2}\longrightarrow0.
\]
Proposition~\ref{prop:w_to_zero} gives $\dim\nu=1$, a
contradiction. This proves the theorem.
\end{proof}

For completeness we record that Corollary~\ref{cor:self-similar-sets} follows from Theorem~\ref{thm:main} exactly as \cite{H}*{Corollary 1.2} from \cite{H}*{Theorem 1.1}.


\end{document}